\documentclass[12pt]{article}
\usepackage{amssymb, amsthm}
\usepackage{pdfsync}
\usepackage[tbtags]{amsmath}
\usepackage{hyperref}

\newcommand {\bv}{{\bf{v}}}
\usepackage{hyperref}

\newcommand {\bk}{{\bf{k}}}
\newcommand {\lt}{\left}
\newcommand {\tum}{\textup{mod}}
\newcommand {\rt}{\right}

\newcommand {\br}{\bf{R}}
\newcommand {\rk}{r_{{\bf{k}}}}
\newcommand {\tk}{\theta_{{\bf{k}}}}
\newcommand {\Thk}{\Theta_{{\bf{k}}}}
\newcommand {\pk}{\psi_{{\bf{k}}}}
\newcommand {\mk}{{\mathcal{M}_{\bf{k}}}}

\newcommand {\Sk}{{\mathcal{S}_{\bf{k}}}}
\newcommand {\ak}{{\mathcal{A}_{\bf{k}}}}
\newcommand {\al}{{\mathcal{A}_\lambda}}
\newcommand {\Bk}{{\mathcal{B}_{\bf{k}}}}
\newcommand {\Bl}{\mathcal{B}_\lambda}
\newcommand {\Ck}{{\mathcal{C}_{\bf{k}}}}
\newcommand {\Dk}{{\mathcal{D}_{\bf{k}}}}
\newcommand {\gk}{\gamma(\bf{k}^2)}
\newcommand {\nk}{N(\bf{k}^2)}

\newcommand {\tho}{\theta_0}

\newcommand {\tn}{t_\bk}
\newcommand {\str}{\stackrel{\mathcal{L}}{=}}
\newcommand {\Ht}{H^{\theta_0}}
\begin{sloppypar}
   
\end{sloppypar}

\newtheorem{cor}{Corollary}[section]

\newtheorem{thm}{Theorem}[section]
\newtheorem{lemma}{Lemma}[section]
\newtheorem{remark}{Remark}[section]

\begin{document}

\title{\bf Heavy tails and one-dimensional localization.}

\author{ M.  Cranston , \, S. Molchanov,\,   N. Squartini} 
\maketitle
\footnote{Research of the authors supported in part by a 
grant from NSF.} 
\begin{abstract}
We address the fundamental questions concerning the operator 
\begin{eqnarray*}
H^{\theta_0}\psi(x)=-\psi''(x)+V(x,\omega)\psi(x),\,\psi(0)\cos\theta_0-\psi'(0)\sin\theta_0=0.
\end{eqnarray*}
where the random potential $V$ has a variety of forms. In one example, it is composed of width one bumps of random heights where the square root of the heights are in the domain of attraction of a stable law with index $\alpha\in(0,1)$ or in another it is composed of width one bumps of height one where the distance  between bumps is in the domain of attraction of a stable law with index $\alpha\in(0,1).$ We consider the existence of Lyapunov exponents, integrated density of states and the nature of the spectrum of the operator.\\


\end{abstract}

\section{Introduction}

In this paper we address a question posed several years ago by G. Zaslovski: what is the effect of heavy tails of one-dimensional random potentials on the standard objects of localization theory:  Lyapunov exponents, density of states, statistics of eigenvalues, etc. ? Professor G. Zaslovski always expressed a special interest in the models of chaos containing strong fluctuations, e.g. L\'{e}vy flights. We'll consider several models of potentials constructed by the use of $iid$ random variables which belong to the domain of attraction of the stable distribution with parameter $\alpha<1.$ In order to put  our results in context, we'll recall the "regular theory" as presented in Carmona-Lacroix \cite{CL} or Figotin-Pastur  \cite{FP}. Consider the one-dimensional Schr\"{o}dinger operator on the half line with boundary condition: 
\begin{eqnarray}\label{seqn}
H^{\theta_0}\psi(x)=-\psi''(x)+V(x,\omega)\psi(x),\,\psi(0)\cos\theta_0-\psi'(0)\sin\theta_0=0.
\end{eqnarray}
where for each $x\in [0,\infty),\,V(x,\cdot)$ is a random variable on a basic probability space $(\Omega,\mathcal{F}, P)$ and $\theta_0\in[0,\pi]$ is fixed. Our potentials $V(x,\omega)$ will be piecewise constant, these are the so-called Kr\"{o}nig-Penny type potentials.

In place of the energy parameter $\lambda,$ we'll sometimes work with the frequency ${\bf{k}}=\sqrt{|\lambda|}.$ Denote the solution of $H^{\theta_0}\psi=\bk^2\psi$ by  $\pk.$ For the solution $\psi_{\bk}$ of this equation
introduce the phase, $\tk,$  and magnitude, $\rk,$  by the Pr\"ufer formulas:
\begin{eqnarray}\label{prufer}
\begin{split}
\pk(x)=&\rk (x)\sin\tk(x)\\
\pk'(x)=&{\bk} \rk(x)\cos\tk(x).
\end{split}
\end{eqnarray}
Then the equation 
\begin{eqnarray}
-\psi_\bk''(x)+V(x,\omega)\psi_\bk(x)=\bk^2\psi_\bk(x)
\end{eqnarray}
can be written, using 
\[Y_\bk(x)=\left(\begin{array}{ll}
\pk(x)\\
\bk^{-1}\pk'(x)
\end{array}\right)\]
and
\[A_\bk(x)=\left(\begin{array}{ll}
0&\bk\\
\frac{V(x)-\bk^2}{\bk}&0
\end{array}\right)\]
in the form $Y'_\bk(x)=A_\bk(x)Y_\bk(x).$ Setting 
$$\Theta_\bk(x)=Y_\bk(x)||Y_\bk(x)||^{-1}$$ and noting
 \[\frac12 (r_\bk^2(x))'=\lt<\Theta_\bk(x),A_\bk(x)\Theta_\bk(x)\rt>r_\bk(x)\]
 and combining this with  (\ref{seqn}), one easily derives the standard Ricatti equations for the phase and magnitude:
\begin{eqnarray}
\begin{split}\label{ricatti}
\tk'(x)=&{\bk}-{\frac{1}{{\bk}}}V(x,\omega)\sin^2\tk(x)\\
\rk'(x)=&\frac12 \rk(x)V(x,\omega)\sin 2\tk(x).
\end{split}
\end{eqnarray}
The monodromy operator, which is in $SL(2,\br),$ denoted by $\mk,$ transforms the values of the solution as follows:
for $a<b,$
\begin{eqnarray}
\begin{split}\label{transfer}
\mk([a,b])\left(\begin{array}{ll}
\pk(a)\\
\bk^{-1}\pk'(a)
\end{array}\right)=&\left(\begin{array}{ll}
\pk(b)\\
\bk^{-1}\pk'(b)
\end{array}\right).
\end{split}
\end{eqnarray}
We'll use the notation $\mk(n)\equiv \mk([n-1,n]),\,n=1,2,....$ 

The standard assumptions on $\mk$ are
\begin{itemize}
\item For all $n,\,E\lt[\ln^+ ||\mk(n)||\rt]<\infty.$
\item The sequence $\{\mk(n):n\ge 1\}$ is either independent or forms a stationary sequence with rapidly decaying correlations.
\item  The random variables $\mk(n)$ have a density in $SL(2,\br)$ with respect to Haar measure, alternatively, the sequence of phases $\{\tk(n):n\ge1\}$ defined by 
\[(\sin \tk(n+1),\cos \tk(n+1))\equiv \Thk(n+1) =\frac{\mk(n)\Thk(n)}{||\mk(n)\Thk(n)||}\]
satisfies the D\"{o}blin condition. That is the Markov chain $\{\tk(n):n\ge1\}$ has a "good" transition density $p_\bk(\theta,\eta).$
 \end{itemize}
Under these assumptions, there exists a continuous and strictly  positive Lyapunov exponent
\[\gk=\lim_{n\to\infty}\frac1n\ln ||\Pi_{j=1}^n\mk(j)||, \,a.s.\]
and an integrated density of states
\[\nk=\lim_{n\to\infty}\frac{\tk(n)}{\pi n},\,\, a.s..\]

The integrated density of states and Lyapunov exponent are related by the Thouless formula, see e.g. \cite{CL},
\[\gk=\gamma_0(\bk^2)+\int_{\br}\ln |\bk^2-u|(N(du)-N_0(du))\]
where $\gamma_0(u)=\sqrt{u}$ and $\bf{N_0}(u)=\frac{\sqrt{u}}{\pi}$ are the Lyapunov exponent and integrated density of states for the operator $H_0=-\frac{d^2}{dx^2}.$ Finally, for a.e $\theta_0$ with respect to Lebesgue measure,  the spectrum of $\Ht$ is pure point  $a.s.$ and the eigenfunction corresponding to the eigenvalue $\bk^2$ is exponentially decreasing at rate $\gamma(\bk^2)$ which is referred to as the exponential localization theorem. Typical references for these results would be \cite{D},  \cite{DLS}, \cite{K}, \cite{SW}, and an overview of the field is contained in the lecture notes \cite{M}. 

In contrast to the preceding classical situation, we'll show that for heavy-tailed potentials, of Kr\"{o}nig-Penny type, (meaning the tail of the distribution of the random potential decays slowly) the limits defining  either $\nk$ or $\gk$ do not exist under the usual normalization. However, after appropriate non-linear normalization, they converge in distribution to non-degenerate random variables. Under appropriate assumptions on the tails, these random variables have a stable distribution. The underlying phenomenon is related to Darling's Theorem, \cite{DD},  on  the contribution of the maximal term to a sum of random variables in the domain of attraction of a stable law with index $\alpha\in(0,1).$ Namely, if $\{\zeta_n\}_{n\ge1}$ are iid, nonnegative  random variables with tails given by $P(\zeta_1>x)=\frac{L(x)}{x^\alpha}$ where $L$ is a slowly varying function, $S_n=\sum_{j=1}^n\zeta_j$ and $\zeta^*_n=\max_{1\le j\le n}\zeta_j$ then the ratio $\zeta_n^*/S_n$ has a nondegenerate limiting law. The quantities we examine have the form of $S_n$ with summands in the domain of attraction of a stable law with index $\alpha\in(0,1).$ Such sums can not be normalized to converge $a.s.$ to a nonzero deterministic constant since from time to time a new summand will have the same order of magnitude as the entire sum. This prevents the usual self averaging we see in the classical case of ergodic potentials. It's important to emphasize that our potentials are not ergodic, so the usual theorems that apply to Schr\"{o}dinger operators with ergodic potentials do not apply in our models. For example, in ${\bf{Model \,III}}$ we have a Lyapunov exponent which is identically zero. If the potential were ergodic, this would imply the spectrum is a.s. absolutely continuous, see \cite{D}, Theorem $4.$ In our model we find the spectrum ${\bf{Model \,III}}$ is a.s. pure point in spite of having a vanishing Lyapunov exponent. We now give a description of four models that will be covered in this paper and the interesting effects that they display.

$${\bf{Model\,I}}$$
In the first model,  $V$ has the form
\begin{eqnarray}
V(x,\omega)=\sum_{n=0}^\infty 1_{[n,n+1)}(x)X_n(\omega)
\end{eqnarray}
where $\{X_n:n\ge0\}$ are $iid$ random variables with common density $p.$ This density will be assumed to be bounded, continuous and satisfy $p(x)>0$ for $x>0$ and vanishing identically for $x\le 0.$ Moreover, we shall assume that $\sqrt{X_n}$ belongs to the domain of attraction of an $\alpha$-stable  law, denoted $St_\alpha,$ where $0<\alpha<1.$ This assumption is equivalent to requiring that $P(\sqrt{X_n}>x)= \frac{L(x)}{x^\alpha}, \,x\to \infty,$  with $L$ a slowly varying function. In ${\bf{Model\,I}}$ we will show that for a.e. $\theta_0,\,P\,\,a.s.,$ the spectrum of $H^{\theta_0}$ is pure-point and the eigenfunctions decay super-exponentially. Also, it turns out that a properly normalized integrated density of states $\nk$ exists and is continuous in spite of the fact that $E[\ln||\mk(n)||]=\infty.$ Perhaps the most interesting result here concerns the Lyapunov exponent. We'll demonstrate that
\[\ln||\mk([0,n])||=\ln ||\mk(1)\mk(2)\cdots\mk(n)||\sim \sum_{j=1}^n\sqrt{X_j}.\]
As a consequence, we'll have
\[\lim_{n\to\infty}\frac{1}{n}\ln||\mk([0,n])||=\infty, \,P-a.s.,\]
whereas
\[\lim_{n\to\infty}\frac{1}{n^{1/\alpha}}\ln||\mk([0,n])||\stackrel{\mathcal{L}}{=}\zeta_\alpha, \]
where $\zeta_\alpha$ has an $St_\alpha$ distribution, where $\mathcal{L}$ is used to indicate convergence in law.
The convergence in distribution can not be improved to $a.s.$ since
\[\zeta(n)\equiv\frac{\ln||\mk([0,n])||}{n^{1/\alpha}}\]
randomly oscillates as $n\to\infty$ while its distribution is tending to the $St_\alpha$ law. This can be seen by the decomposition
\begin{eqnarray}
\begin{split}
\zeta(2^{n+1})\sim &2^{-(n+1)/\alpha}\sum_{j=1}^{2^{n}}\sqrt{X_j}+2^{-(n+1)/\alpha}\sum_{j=2^n+1}^{2^{n+1}}\sqrt{X_j}\\
\sim& 2^{-1/\alpha}(\zeta(2^{n})+\tilde\zeta(2^{n}))
\end{split}
\end{eqnarray}
where $\zeta(2^{n})$ and $\tilde\zeta(2^n)$ are independent and nearly $St_\alpha$ distributed. Thus, $\zeta(2^{n+1})-\zeta(2^{n})= (2^{-1/\alpha}-1)\zeta(2^{n})+2^{-1/\alpha}\tilde\zeta(2^{n})$ and so $a.s.$ convergence does not hold in this situation.
The results are summarized in the following theorem. 

\begin{thm}\label{MI}
In ${\bf{Model\,I}},$ for any $\bk^2,$

the standard integrated density of states exists and is given by
\[{\nk}=\lim_{n\to\infty}\frac{\tk(n)}{\pi n},\,\,P-a.s..\]

In addition, the linear scale Lyapunov exponent is infinite, i.e.
\[\gamma(\bk^2)=\lim_{n\to\infty}\frac{1}{n}\ln||\mk([0,n])||=\infty, \,\,P-a.s.,\]
whereas in the nonlinear scale, we have convergence in distribution,
\[\gamma_{\bf{nl}}(\bk^2)=\lim_{n\to\infty}\frac{1}{n^{1/\alpha}}\ln||\mk([0,n])||\stackrel{\mathcal{L}}{=}\zeta_\alpha, \]
where $\zeta_\alpha$ has an $St_\alpha$ distribution.

For a.e. $\tho\in [0,\pi],$ the operator $\Ht$ has pure point spectrum a.s.. In addition, the eigenfunctions satisfy
\[\lim_{x\to\infty} \frac{1}{2 x^{1/\alpha}}\ln \lt(\pk(x)^2+\frac{1}{\bk^2}\pk'(x)^2\rt)\str-\xi_\alpha.\]
\end{thm}


$${\bf{Model\,II}}$$
This model is the same as ${\bf{Model\,I}}$ except now the potential is non-positive. We take 
\begin{eqnarray}
V(x,\omega)=-\sum_{n=0}^\infty 1_{[n,n+1)}(x)X_n(\omega)
\end{eqnarray}
with the same conditions on the distribution of $X_n$ as in ${\bf{Model\,I}},$ namely,  the random variables $X_n$ have common density $p$ which is bounded, continuous and satisfies $p(x)>0$ for $x>0,\,p(x)\equiv 0$ for $x\le 0$ and $P(\sqrt{X_n}>x)= \frac{L(x)}{x^\alpha},$  where $L$ is a slowly varying function.
The corresponding operator is essentially self-adjoint without any assumptions on the tails of $X_n.$ This last fact can't be proven by appealing to Weyl's criterion (which requires the condition $V(x)\ge-c_0-c_1x^2,\,c_0>0,c_1>0$ as can be found in \cite{LS}.) However, it does follow from a Theorem of Hartman which is stated below. In this model we'll switch back to $\lambda$ instead of $\bk^2$ since there is spectrum on both sides of the origin in this case. The Lyapunov exponent $\gamma(\lambda)$ is positive and continuous (at least in the situation of ${\bf{Model\,I}}$ when the tails aren't too heavy.) However, in   ${\bf{Model\,II}}$ the density of states $N(\lambda)$ demonstrates unusual behavior:
\begin{thm}\label{MII}
In ${\bf{Model\,II}},$ for any $\lambda,$
\begin{eqnarray}
\theta_\lambda(n)=\sum_{j=0,\,X_j\ge -\lambda}^n\sqrt{X_j+\lambda}+O(n), \,\,\,P-a.s..
\end{eqnarray}
Consequently,
\[\lim_{n\to\infty}\frac{\theta_\lambda(n)}{\pi\,n}=\infty,\,\,\,P-a.s.,\]
but
\[\lim_{n\to\infty}\frac{\theta_\lambda(n)}{\pi\,n^{1/\alpha}}\stackrel{\mathcal{L}}{=}\zeta_\alpha,\]
where $\zeta_\alpha$ has an $St_\alpha$ distribution.
Finally,
\begin{eqnarray}
\gamma(\lambda)=\lim_{n\to\infty} \frac{1}{n}\ln ||\mathcal{M}_\lambda([0,n])||,\,\,P\,\,a.s.
\end{eqnarray}
and $\gamma(\lambda)>0.$
Consequently, $H^{\theta_0}$ has pure point spectrum for $a.e.\, \theta_0.$ In addition, the eigenfunctions satisfy
\[\lim_{n\to\infty} \frac{1}{ x}\ln r_\lambda(x)=\lim_{x\to\infty} \frac{1}{2 x}\ln \lt(\psi_\lambda(x)^2+\frac{1}{\lambda}\psi_\lambda'(x)^2\rt)=-\gamma(\lambda).\]

\end{thm}

$${\bf{Model\,\,\,III}}$$
For this model, $\{Y_n:n\ge1\}$ are $iid$ random variables with common density $p$ which is bounded, satisfies $p(x)>0$ for $x>0$ and $P(Y_1>x)=\frac{L(x)}{x^\alpha},\,\alpha\in(0,1),$ for some slowly varying function $L.$
Set 
\begin{eqnarray}\label{eln}
L_n=S_n+n=\sum_{j=1}^nY_j+n
\end{eqnarray}
and define 
\[V(x,\omega)=\sum_{j=1}^\infty {\bf{1}}_{[L_j(\omega)-1,\,L_j(\omega)]}(x).\]
This potential is a system of bumps of height  
$1$ and width $1$ with the $j^{th}$ and $(j+1)^{st}$ bumps being separated by the random distance $Y_{j+1}.$ This could be generalized easily to bumps of height  $h$ and width $\delta>0,$ with no complications. Under the assumption that $\int_x^\infty p(y)dy= \frac{L(x)}{x^\alpha},$ with $L$ a slowly varying function, the distance between bumps exhibits strong fluctuations. Observe that the right edge of the $n^{th}$ bump has distance $ S_n+n \equiv L_n$ from the origin. Under our assumption on the tail behavior of $p(x),$ it follows that 
\[\frac{L_n}{n^{1/{\alpha}}}\stackrel{\mathcal{L}}{\to}\zeta_{\alpha},\]
where, as before, $\zeta_{\alpha}$ has a $St_{\alpha}$ distribution. 
For this model we have the following results.

\begin{thm}\label{MIII}
In ${\bf{Model\,III}},$ 
$$\tk(L_n)=\bk\sum_{j=1}^nY_j+O(n).$$
The integrated density of states exists but,
$$\nk=\lim_{x\to\infty} \frac{\tk(x)}{\pi\,x} =\infty,\,\,a.s.,$$  
while 
$$\lim_{x\to\infty} \frac{\tk(x)}{\pi\,x^{1/\alpha}}\stackrel{\mathcal{L}}{=}\bk\zeta_\alpha. $$ 
The Lyapunov exponent in the linear scale is
\begin{eqnarray}\label{zerolyp}
\gk=\lim_{n\to\infty}\frac{1}{L_n}\ln||\mk([0,L_n])||=0
\end{eqnarray}
while there exists a Lyapunov exponent in the 'non-linear' scale
\[\gk_{nl}=\lim_{n\to\infty}\frac{1}{n}\ln||\mk([0,L_n])||>0\]
and
\[\lim_{n\to\infty}\frac{1}{L^\alpha_n}\ln||\mk([0,L_n])||\stackrel{\mathcal{L}}{=}\frac{\gk_{nl}}{\zeta_\alpha^\alpha}.\]
For $a.e.\, \theta_0,\,H^{\theta_0}$ has pure point spectrum a.s.. 
Moreover, if  $\bk^2\in \Sigma(H^{\theta_0})$ then the corresponding eigenfunction satisfies
\[\lim_{x\to\infty}\frac{1}{2x^{\alpha}}\ln\lt(\psi_\bk^2(x)+\psi_\bk'^2(x)\rt)\stackrel{\mathcal{L}}{=}-\frac{\gk_{nl}}{\zeta_\alpha^\alpha}.\]
\end{thm}
$${\bf{Model\,\,\,IV}}$$
The final model incorporates the strong fluctuations in the bump size from ${\bf{Model\,\,\,I}}$ and ${\bf{Model\,\,\,II}}$ as well as the strong fluctuation in the gap size as in ${\bf{Model\,\,\,III}}.$ The results for this model are similar to those for the first three models and so we won't state the results here as a Theorem but rather confine ourselves to a description of what can be proven. In this model we will denote the bump sizes by the $iid$ random variables $\{X_n\}$ and the gap sizes by $\{Y_n\}$ again assumed to be $iid.$ As before we will assume that $\{\sqrt{X_n}\}$ and $\{Y_n\}$ have densities $p_i(x),\,i=1,2$ respectively, which are bounded, positive for $x>0$ and have tail behavior $\int_x^\infty p_i(y)dy=\frac{L_i(x)}{x^{\alpha_i}},\,\,i=1,2,\,\,\alpha_1,\alpha_2\in (0,1)$ where the $L_i$ are slowly varying
 as above so that the random variable $\{\sqrt{X_n}\}$ and $\{Y_n\}$ are in the domain of attraction of stable laws,
\[\lim_{n\to\infty}\frac{1}{n^{1/\alpha_1}}\sum_{j=1}^n\sqrt{X_j} \stackrel{\mathcal{L}}{=}\zeta_{\alpha_1}\]
and
\[\lim_{n\to\infty}\frac{1}{n^{1/\alpha_2}}\sum_{j=1}^n Y_j \stackrel{\mathcal{L}}{=}\zeta_{\alpha_2}\]
where $\zeta_{\alpha_1}$ and $\zeta_{\alpha_2}$ have respectively $St_{\alpha_1}$ and $St_{\alpha_2}$ distributions. The explicit form of the potential here is 
\begin{eqnarray}
V(x)=\sum_{n=1}^\infty (-1)^{\epsilon_n}X_n1_{[L_n-1, L_n]}(x),
\end{eqnarray}
where $L_n=\sum_{j=1}^n Y_j+n$ and $\{\epsilon_n:n\ge 1\}$ is an iid sequence of Bernoulli random variables $P(\epsilon_n=1)=P(\epsilon_n=0)=\frac12$ and they are independent of the the sequences $\{X_n:n\ge 1\}$ and $\{Y_n:n\ge 1\}.$
In this model, again switching back to $\lambda$ instead of $\bk^2$ due to the appearance of spectrum to the right of the origin. 
A typical result is that 
\[\ln ||\mathcal{M}_\lambda([0,L_n])||\sim \sum_{j=1}^n\sqrt{X_j}\]
and
\begin{eqnarray}
\begin{split}
\frac{\ln ||\mathcal{M}_\lambda([0,L_n])||}{L_n}\sim&\frac{\sum_{j=1}^n\sqrt{X_j}}{\sum_{j=1}^n Y_j}\\
=&\frac{n^{-1/\alpha}\sum_{j=1}^n\sqrt{X_j}}{n^{-1/\alpha}\sum_{j=1}^n Y_j}\\
\stackrel{\mathcal{L}}{\to}&\frac{\zeta_{\alpha_1}}{\zeta_{\alpha_2}},
\end{split}
\end{eqnarray}
where $\zeta_{\alpha_1}$ and $\zeta_{\alpha_2}$ are independent with $St_{\alpha_1},\,St_{\alpha_2}$ distributions, respectively. 
The rotation number takes the form
\[\theta_\lambda(n)=\lambda \sum_{i=1}^nY_i+\sum_{i=1,\epsilon_i=-1}^n\sqrt{X_i+\lambda}+O(n).\]
The first sum on the right hand side is the accumulated rotation across the gaps as in ${\bf{Model\,\,III}}.$ The second term is result of the rotation across the negative bumps as in ${\bf{Model\,\,II}}.$ The $O(n)$ term is the effect of the positive bumps as in ${\bf{Model\,\,I}}.$
Thus, there is a nonlinear version of the integrated density of states which depends on which index $\alpha_1$ or $\alpha_2$ is smaller. 
Setting $\alpha=\alpha_1\wedge\alpha_2,$ we have the limit 
\[N_{nl}(\lambda)=\lim_{n\to\infty}\frac{\theta_\lambda(n)}{\pi\,\,n^{1/\alpha}}\stackrel{\mathcal{L}}{=}\zeta_\alpha,\]
where as usual, $\zeta$ is a random variable with a $St_\alpha$ distribution. 
These follow in a manner similar to the results for ${\bf{Models \,\,\,I,\,\,II,\,\,III}}$ so we omit the proofs in this case.
As mentioned above, in place of the energy parameter $\lambda,$ in ${\bf{Model\,I}}$ and ${\bf{Model\,III}}$ we'll work with the frequency ${\bf{k}}=\sqrt{|\lambda|}.$ For the solution $\psi_{\lambda}$ of the boundary value problem 
\[H^{\theta_0}\pk={\bf{k}}^2\psi_{\bf{k}}\]
the phase, $\theta_\bk,$  and magnitude, $r_\bk,$  are given by the Pr\"ufer formulas, see (\ref{prufer}), 
and satisfy the standard Ricatti equations given at (\ref{ricatti}).
Recall that we denote by $\mathcal{M}_{\lambda}([a,b])$  the propagator of the system whose action is given in (\ref{transfer}). 

In the case of ${\bf{Model\, I}},$ the propagator is the product of random matrices,
\[\mk([0,n])=\ak(X_n)\cdots\ak(X_2) \ak(X_1),\]
where we have slightly changed the notation used in the introduction for the matrices in this product.
The form of these matrices depends on whether $X_l<\bk$ or $X_l\ge\bk.$ In the former case,
\begin{eqnarray}
\begin{split}\label{trans1}
\ak(X_l)=\left(\begin{array}{ll}\cos\sqrt{\bk^2-X_l}&\frac{\bk}{\sqrt{\bk^2-X_l}}\sin\sqrt{\bk^2-X_l}\\
-\frac{\sqrt{\bk^2-X_l}}{\bk}\sin\sqrt{\bk^2-X_l}&\cos\sqrt{\bk^2-X_l}\end{array}\right),
\end{split}
\end{eqnarray}
whereas in the latter, 
\begin{eqnarray}
\begin{split}\label{trans2}
\ak(X_l)=\left(\begin{array}{ll}\cosh\sqrt{X_l-\bk^2}&\frac{\bk}{\sqrt{X_l-\bk^2}}\sinh\sqrt{X_l-\bk^2}\\
\frac{\sqrt{X_l-\bk^2}}{\bk}\sinh\sqrt{X_l-\bk^2}&\cosh\sqrt{X_l-\bk^2}\end{array}\right).
\end{split}
\end{eqnarray}

In the case of ${\bf{Model\, II}},$ again  the propagator is the product of random matrices,
\[\mk([0,n])=\Bk(X_n)\cdots\Bk(X_2) \Bk(X_1),\]
which are simply
\begin{eqnarray}
\begin{split}\label{trans3}
\Bl(X_l)=\left(\begin{array}{ll}\cos\sqrt{\lambda+X_l}&\frac{\sqrt{|\lambda|}}{\sqrt{\lambda+X_l}}\sin\sqrt{\lambda+X_l}\\
-\frac{\sqrt{\lambda+X_l}}{\sqrt{|\lambda|}}\sin\sqrt{\lambda+X_l}&\cos\sqrt{\lambda+X_l}\end{array}\right),
\end{split}
\end{eqnarray}
for $X_l + \lambda\ge0$ with a similar matrix using hyperbolic trig functions when $X_l + \lambda<0$ as in ${\bf{Model\, I}},$ namely
\begin{eqnarray}
\begin{split}\label{trans2}
\al(X_l)=\left(\begin{array}{ll}\cosh\sqrt{-(\lambda+X_l)}&\frac{|\lambda|}{\sqrt{-(\lambda+X_l)}}\sinh\sqrt{-(\lambda+X_l)}\\
\frac{\sqrt{-(\lambda+X_l)}}{|\lambda|}\sinh\sqrt{-(\lambda+X_l)}&\cosh\sqrt{-(\lambda+X_l)}\end{array}\right).
\end{split}
\end{eqnarray}
For ${\bf{Model\, III}},$   the propagator is more easily expressed at the times $L_n=Y_1+\cdots+Y_n+n=S_n+n$ marking the end of the $n^{th}$ gap,
then it is the product
\[\mk([0,L_n ])=\Ck(Y_n)\cdots\Ck(Y_2) \Ck(Y_1),\]
where $\Ck(Y_l)$ factors into the product of the transfer matrix between the bumps multiplied by the transfer matrix across the following bump,
\begin{eqnarray}
\begin{split}
\Ck(Y_l)=\tilde\Ck(Y_l)\hat\Ck(Y_l).
\end{split}
\end{eqnarray}
The monodromy operator between the bumps is 
\begin{eqnarray}
\begin{split}\label{trans4}
\hat{\Ck}(Y_l)=\left(\begin{array}{ll}\cos \bk Y_l&\sin \bk Y_l\\
-\sin\bk Y_l&\cos\bk Y_l\end{array}\right),
\end{split}
\end{eqnarray}
while the form of the operator across the bumps ,  $\tilde\Ck(Y_l),$ is deterministic and its form depends on whether $1<\bk$ or $1\ge\bk.$ In the former case,
\begin{eqnarray}
\begin{split}\label{trans5}
\tilde\Ck(Y_l)=\left(\begin{array}{ll}\cos\sqrt{\bk^2-1}&\frac{\bk}{\sqrt{\bk^2-1}}\sin\sqrt{\bk^2-1}\\
-\frac{\sqrt{\bk^2-1}}{\bk}\sin\sqrt{\bk^2-1}&\cos\sqrt{\bk^2-1}\end{array}\right),
\end{split}
\end{eqnarray}
whereas in the latter, 
\begin{eqnarray}
\begin{split}\label{trans6}
\tilde\Ck(Y_l)=\left(\begin{array}{ll}\cosh\sqrt{1-\bk^2}&\frac{\bk}{\sqrt{1-\bk^2}}\sinh\sqrt{1-\bk^2}\\
-\frac{\sqrt{1-\bk^2}}{\bk}\sinh\sqrt{1-\bk^2}&\cosh\sqrt{1-\bk^2}\end{array}\right).
\end{split}
\end{eqnarray}
The matrices $\{\Ck(Y_l):l\ge1\}$ are  $iid$ elements of $SL(2,{\bf{R}}).$ 

The monodromy operator for ${\bf{Model\,IV}}$ is a mixture of products of matrices of the above form, across a gap of width $Y_l$ a matrix of the form (\ref{trans4}) is represented in the product. When a bump of height $(-1)^{\epsilon_l}X_l$ is encountered and $\epsilon_l=1,$ a matrix of the form (\ref{trans1}) or (\ref{trans2}) appears in the product depending on whether $X_l<\bk$ or $X_l\ge \bk.$, When $\epsilon_l=-1,$ then a matrix of the form (\ref{trans3}) is entered in the product. 

 \section{Auxilliary Results}

 For the proof of a.s.  localization of $\bf{Models\,I}$ and ${\bf{II}}$ we'll use the following classical result.

 \begin{thm}(Furstenberg) If $\{M_j\}_{j\ge 1}$ are $i.i.d.$ elements of $SL(2,R)$ with $E[\ln||M_1||]<\infty,$ and the corresponding Markov chain 
 \[\theta_n=\frac{M_nM_{n-1}\cdots M_1\theta_0}{||M_nM_{n-1}\cdots M_1\theta_0||}\]
 is ergodic and the distribution of $M_1$ is not contained in  a compact subgroup of $SL(2,R)$ then
 \begin{eqnarray}
 \lim_{n\to\infty}\frac{1}{n}\ln ||M_nM_{n-1}\cdots M_1\theta_0||=\gamma>0,\,a.s..
\end{eqnarray}
Moreover there is a one dimensional subspace $\bf{W}\subset {\bf{R^2}}$ such that for $\bf{\theta}\in W\setminus\{0\},$
\begin{eqnarray}
 \lim_{n\to\infty}\frac{1}{n}\ln ||M_nM_{n-1}\cdots M_1\theta||=-\gamma,\,a.s..
\end{eqnarray}
 \end{thm}
 
An important ingredient in establishing the decay of eigenfunctions is the following result of Sch'nol.
A more general version of the following result is Theorem C.4.1 in \cite{S}.
\begin{thm} (Sch'nol)
Suppose $V$ is one of the potentials in ${\bf{Model\,I-IV}}$ below and $Hu=\Delta u+V u=Eu$ with $u$ polynomially bounded. Then $E\in \Sigma(H).$
\end{thm}
 In the classical setting mentioned in the ${\bf{Introduction}},$ this has been used in conjunction with the Lyapunov exponents for the transfer matrix to establish that a generalized eigenfunction of polynomial growth must actually decay exponentially. This was the technique used in \cite{C}, \cite{K} and  \cite{M} for example. In the case of the heavy tailed potentials examined in this work, it establishes super-exponential decay of the eigenfunctions, that is an eigenfunction will satisfy 
 $$\lim_{x\to\infty}\frac{1}{x^{1/\alpha}}\ln\sqrt{\psi_\bk^2(x)+\psi_\bk'^2(x)}\stackrel{\mathcal{L}}{=}-\zeta_\alpha,$$
 where $\zeta_\alpha$ is a random variable in the class $St_\alpha.$

  In ${\bf{Model\,\,\,II}}$ we use the following theorem to establish the essential self-adjointness of $H.$
\begin{thm}
(Hartman's Theorem)

If one can find a sequence of intervals $\Delta_n=[x_n,x_n+\delta]$ with $\delta>0$ fixed and  $x_n\to\infty$ and a constant $c_0>-\infty$ such that $V(x)\ge c_0\,\mbox{on}\,\,\Delta_n,$ then the operator $H$ is essentially self-adjoint.
\end{thm}
This applies to ${\bf{Model\,\,II}},$ 
\begin{cor}
The operators $H^{\theta_0}$  in  ${\bf{Models\,\,II}}$ and ${\bf{IV}}$ are essentially self-adjoint.
\end{cor}
\begin{proof}
Just use Borel-Cantelli, $\sum_{n=1}^\infty P(-X_n>c_0)=\infty$ so there is a.s. an infinite subsequence $X_{n_k}$ such that 
$-X_{n_k}>c_0$ for all $k.$ Then taking $x_k=X_{n_k}$ and $\delta = 1$ the conditions of Hartman's Theorem are satisfied. Thus the operators defined in ${\bf{Models\,\,II}}$ and ${\bf{IV}}$ are essentially self-adjoint.
\end{proof}

We need the following lemma when establishing that  the spectrum is discrete.
\begin{lemma}
Assume there exist a sequence $\{x_n\}_{n\ge1}$ such that for some $c>1,\,x_n\le c^n,\,n\ge 1$ and  a sequence $\{h_n\}_{n\ge1}$ such that $\sum_{n=1}^\infty e^{-\sqrt{h_n}\delta_n}<\infty,$ where $x_n,\,h_n,\,\delta_n$ satisfy $V(x,\omega)\ge h_n,\,x_n\le x\le x_n+\delta_n, n\ge 1.$ Then for almost all $\theta_0,$ the spectrum of $H^{\theta_0}$ with boundary condition $\cos\theta_0\psi(0)-\sin \theta_0 \psi'(0)=0$ is $a.s.$ pure point.
\end{lemma}

The Lemma follows from general results, see for example \cite{M}. From the assumptions of the Lemma, it follows that  for $\lambda\in {\bf{R}}, \, R_{\lambda}(0,\cdot) \in L^2([0,\infty),dx).$ Together with the existence of a density for the distribution of $\theta_\lambda(1)$ and Kotani, \cite{K}, this implies the conclusion of the Lemma. We now apply this Lemma to show 

\begin{cor}
For $a.e.\,\theta_0, $ the spectrum for  $H^{\theta_0}$  in ${\bf{Model\,I}}$ is $a.s.$ pure point.
\end{cor}

\begin{proof}
Consider the intervals $\Delta_n=[n^2,(n+1)^2]$ which contains $2n+1$ bumps of unit width. Observe that
\begin{eqnarray*}
P\lt(\max_{k\in \Delta_n} X_k\le n\rt)=& \lt(1-\frac{L(n)}{n^{\alpha/2}}\rt)^{(2n+1)}\\
\sim& e^{-2L(n)n^{1-\alpha/2}}.
\end{eqnarray*}
Thus, by Borel-Cantelli, $a.s.$ in $\omega$ there is an $n_0=n_0(\omega)$ such that for each $n\ge n_0$ there is an $x_n\in \Delta_n$ with $V(x_n)\ge n.$ Setting $\delta_n=1,\, h_n=n$ an application of Lemma $3.1,$ shows that the spectrum is pure point a.s..
\end{proof}

Another essential and classic result is the following due to Kotani, \cite{K}.

\begin{thm} If $H^{\theta_0}$ is an essentially self-adjoint operator on $[0,\infty)$ with boundary condition as given in (\ref{seqn}) and for $a.e. \,\lambda\in{\bf{R}}$ there exists a solution $\psi_\lambda\in L^2([0,\infty))$ of the equation $H\psi=\lambda\psi,$ then for $a.e. \,\theta_0\in[0,2\pi]$ the spectrum of the operator $H^{\theta_0}$ is pure point.
\end{thm}   
\[{\bf{Remark}}\]
Kotani's result will provide the localization result for $H^{\theta_0}$ for $a.e.\, \theta_0\in[0,2\pi]$ for all of our models. Note that the spectrum of $H^{\theta_0}$ is $P\,a.s.$ the full energy axis in the ${\bf{Models\,II,IV}}$ and $[0,\infty)$ for  ${\bf{Models\,I,III}}$ with a possible single isolated negative eigenvalue. For such kinds of "solid spectra' the theorem of A. Gordon states that there exists  $G_\delta$ set $\Gamma\subset [0,\pi]$  with $m(\Gamma)=0$ so that for every $\theta_0\in \Gamma,$ the spectrum of $H^{\theta_0}$ is our singular continuous on $(-\infty,\infty)$ in the case of ${\bf{Models\,II,IV}}$ and $(0,\infty)$ for ${\bf{Models\,I,III}}$.

We now give another version of Sch'nol's  Theorem, \cite{SS},  which can be used to determine the nature of the spectrum. Denote the spectral measure of $H$ by $\rho.$
\begin{thm}
Assume the potential $V$ satisfies that there exists a sequence of points $y_n$ with $\lim_{n\to\infty}y_n=\infty$ and numbers $\delta>0,\,c>0$ such that for all $n,\,V(x)\ge -c,\,x\in(y_n-\delta,y_n+\delta).$  Denote the solutions of
\begin{eqnarray}
\begin{split}
-y''+Vy=&\lambda\,y,\\
y(0)=&0,\,y'(0)=1.
\end{split}
\end{eqnarray}
by $\psi_\lambda.$
Then, $\forall \epsilon>0, \exists c_\epsilon>0$ such that
\begin{eqnarray*}
|\psi_\lambda(x)|\le c_\epsilon x^{\frac12+\epsilon}, \mbox{for}\, \rho \, a.e.\, \lambda.
\end{eqnarray*}
\end{thm}
Shnoll's Theorem implies, as observed in as yet unpublished work of Gordon and Molchanov, that given any sequence $\{x_n\}$ with $x_n\to \infty,$ 
\begin{eqnarray}\label{shnol}
|\psi_\bk(x_n)|\le c_\epsilon(\lambda)n^{\frac12+\epsilon},\,\mbox{for} \, \rho\, a.e. \,\lambda.
\end{eqnarray}  
This is derived from the inequality using the spectral measure $\rho,$
\[\rho\{\lambda:|\psi_\lambda(x_n)|\ge n^{1/2+\epsilon}\}\le \frac{\int_0^a|\psi_\lambda(x_n)|^2\rho(d\lambda)}{n^{1/2+\epsilon}}\]
followed by an application of the Borel-Cantelli Lemma.

This observation can be used to show that with very large gaps between the bumps in ${\bf{Model\,III}}$ the spectrum becomes singular absolutely continuous a.s.. 

\begin{cor}
In ${\bf{Model\,III}}$ if $P(X_n>x)\sim\frac{c}{\ln^{1+\delta} x},\,\,x\to\infty$ for some $\delta>0,$ then the spectral measure $\rho$ is purely singular continuous, i.e. is singular with respect to Lebesgue measure with no eigenvalues.
\end{cor}
\begin{proof}
 By Furstenburg's Theorem, there is a solution $\psi_\lambda^+$ for which $\psi_\lambda^+(L_n)$ is growing exponentially, where $L_n$ is defined at (\ref{eln}). Take $x_n=L_n$ in (\ref{shnol}), so that by the above consequence, (\ref{shnol}),  of Shnoll's Theorem, the spectral measure is singular with respect to Lebesgue measure.  Also, the weaker inequality $|\psi_\lambda(L_n)|\ge e^{-(\mu(\lambda)+\epsilon)n}$ holds eventually for any solution where $\mu(\lambda)$ is the Lyapunov exponent. Since $P(X_j\ge e^{c j}\,\,i.o.)=1$ it follows that

\[\int_0^{L_n}\psi_\lambda(x)dx \ge  \sum_{j=1}^n X_j e^{-2\mu(\lambda)j}.\]
Thus $\psi_\lambda$ is not an eigenfunction and there are no eigenvalues.

\end{proof}

\section{Proofs}
All results on the integrated density of states ( rotation number) are based on the following elementary variational fact. 

\begin{lemma}
Consider the equation $H^{\theta_0}\psi=\lambda\psi$ on $[0,\infty)$ with the boundary condition set in (\ref{seqn}). Assume that on some interval $[0,l_n]$ the potential $V$ is piecewise constant on the subintervals $\Delta=[0,l_1),\,\Delta=[l_1,l_2),\cdots,\Delta_n=[l_{n-1},l_n]$ with constant value $V_i$ on $\Delta_i,\,i=1,2,\cdots,n.$ Then
\begin{eqnarray}\frac{1}{\pi}\theta_\lambda(l_n)=\sum_{k:V_k<\lambda}\sqrt{\lambda-V_k}(l_k-l_{k-1})+O(n).
\end{eqnarray}
\end{lemma}

\begin{cor}

\begin{itemize}
\item 
For ${\bf{Model\,II}}$ for any real $\lambda$ 
\begin{eqnarray}
\frac{1}{\pi}\theta_\lambda(n)=\sum_{k=1}^n\sqrt{|X_k|}+O(n).
\end{eqnarray}
\item 
For ${\bf{Model\,III}}$ for any real $\lambda$ 
\begin{eqnarray}
\frac{1}{\pi}\theta_\lambda(L_n)=\sqrt{\lambda}\sum_{k=1}^n\sqrt{Y_k}+O(n).
\end{eqnarray}
\end{itemize}
\end{cor}
We first turn attention to ${\bf{Model\,I}}$  and obtain a result on the asymptotic distribution of the Markov chain $\{\tilde{t_\bk}(n):\,n\in\{0,1,2,...\}\}$ for $\tk$ the solution of  (\ref{ricatti}),
where $\tilde{t_\bk}(n)=t_\bk(n)\tum\,\pi.$
Define $t_\bk(n)=\tan \tk(n)$ and $Z_n=\sqrt{X_n-\bk^2}1_{X_n>\bk^2}+\sqrt{\bk^2-X_n}1_{X_n<\bk^2}.$  Using (\ref{trans1})  we obtain
 \begin{eqnarray}
  \tn(n+1)=\frac{\tn(n)+\frac{\bk}{Z_n}\tanh Z_n}{\tn \frac{Z_n}{\bk}\tanh Z_n +1},\,\mbox{if}\,X_n>\bk^2,
\end{eqnarray}
while
 \begin{eqnarray}
  \tn(n+1)=\frac{\tn(n)+\frac{\bk}{Z_n}\tan Z_n}{-\tn(n)\frac{Z_n}{\bk}\tan Z_n +1},\,\mbox{if}\,X_n>\bk^2.
\end{eqnarray}
Define function 
\begin{eqnarray}\label{F}
F(t,y)=\frac{t+\frac{\bk}{y}\tanh y}{t \frac{y}{\bk}\tanh y +1},\,y\ge0,
\end{eqnarray}
and 
\begin{eqnarray}\label{G}
G(t,y)=\frac{t+\frac{\bk}{y}\tan y}{-t\frac{y}{\bk}\tan y +1},0\le y\le \bk.
\end{eqnarray}
Then we have 
$$\tn(n+1)=F(\tn(n),Z_n)1_{X_n>\bk}+G(\tn(n),Z_n)1_{X_n<\bk},$$ 
from which it is clear that $\tn$ forms a Markov chain. Moreover, there is a one-to-one correspondence between $\tn$ and $\tilde{\tk}=\tk \mbox{mod} \pi$ via the tangent mapping. We will use this to prove the process $\tilde{\tk}$ is positive recurrent and has a nice stationary probability distribution and the rate of convergence to this distribution is  exponential.

\begin{thm}\label{posrec}
The Markov chain $\{\tn(n):n\ge0\}$ is  a Harris m-recurrent chain where $m$ is Lebesgue measure on $(-\infty,\infty).$
The chain possesses a stationary probability measure $\nu_{\bf{k}}$ and there is a constant $\gamma\in(0,1)$ and $C>0$ such that for any $t\in(-\infty,\infty),$
\begin{eqnarray}
||P(\tn(n)\in \cdot|\tn(0)=t)-\nu_{\bf{k}}(\cdot)||_{TV}\le C\gamma^n,
\end{eqnarray}
where $||\,\,||_{TV}$ denotes the total variation norm. The measure $\nu_{\bf{k}}$ is absolutely continuous with respect to Lebesgue measure and its density is bounded.

\end{thm}

Since $\tilde{\tk}(n)=\tan^{-1}\tn(n)+\frac{\pi}{2},$ we have the following immediate consequence. 

\begin{cor}
The Markov chain $\{\tilde{\tk}(n):n\ge0\}$ is  a Harris m-reccurent chain where $m$ is Lebesgue measure on $[0,\pi].$
The chain possesses a stationary probability measure $\mu_\bk$ and there is a constant $\eta\in(0,1)$ and $C>0$ such that for any $\theta\in [0,\pi],$
\begin{eqnarray}
||P(\tilde{\tk}(n)\in \cdot|\tilde{\tk}(0)=\theta)-\mu_\bk(\cdot)||_{TV}\le C\eta^n.
\end{eqnarray}
 The measure $\mu_{\bf{k}}$ is absolutely continuous with respect to Lebesgue measure and its density is bounded.

\end{cor}

This corollary implies the existence of the integrated density of sates.,
\[{\nk}=\lim_{n\to\infty}\frac{\tk(n)}{\pi n},\,\,P-a.s..\]
This follows since 

Before we proceed with the proof of Theorem $\ref{posrec},$ we recall a few facts about Harris recurrence from the paper {B1}, \cite{B2}. The following conditions imply the existence of a stationary probability measure for $\{\tn(n):n\ge 0\}.$ Denote the  Borel sets in ${\bf{R}}$ by $\mathcal{B}.$

$${\bf{Condition \,\,A}}$$

There exists a $C\in \mathcal{B}$ and a probability measure $\nu$ and a constant $\beta>0$ such that 
\begin{eqnarray}
P(\tn(1)\in A|\tn(0)=t)\ge\beta\nu(A),\,\,\forall t\in C\,\mbox{ and} \,A\in\mathcal{B} .
\end{eqnarray}

$$\bf{Condition\,\, B}$$

There exists a measurable function $W:{\bf{R}}\to[1,\infty)$ and constants $\lambda<1$ and $K<\infty$ such that 
\begin{eqnarray}
PW(t)\le \lambda W(t)1_{C^c}(t)+K1_{C}(t).
\end{eqnarray}

$$\bf{Condition\,\, C}$$

There exists a $\gamma>0$ such that $\beta\nu(C)\ge \gamma.$\\
\vspace{.25in}

Once we establish Conditions A, B and C for ${\bf{Model\,\,I}}$ we'll have from \cite{B2} that Theorem (\ref{posrec}) holds.
\begin{proof}(Theorem\,\ref{posrec})

The following condition will imply $\bf{Condition \,\,A},$ and we refer to \cite{A} page $151$ for the proof. 

$$\bf{Condition\,\, D}$$

Suppose the set $C\subset \bf{R}$ is such that the density $p(t,s)$ for $\{\tn(n):\,n\ge 0\}$ satsifies for some $d_{\bk}>0$
$$ p(t,s)\ge d_\bk,\,\,t,s\in C$$
and that $C$ is a recurrent set. Then $\bf{Condition\,\,A}$ holds with $\nu(B)=m(B\cap C)/m(C)$ where $m$ is Lebesgue measure. \\
\vspace{.2in}

Notice that $\bf{Condition\,\,C}$ holds with this choice of $\nu$ since $\nu(C)=1.$
\begin{lemma}
$\bf{Condition\,\, D}$ holds
\end{lemma}
\begin{proof}
For Condition D, set $C=[\frac{1}{\sqrt 3},1].$ 
Then $ p(t,s)\ge d_\bk,\,\,t,s\in C$ will be established using the function $F$ and $G$ defined by (\ref{F}) and (\ref{G}). Now for ${\bf{Model\,I}},$ the Ricatti equation (\ref{ricatti}) is 
 \begin{eqnarray}
 \begin{split}\label{thetaeqn}
 \tk(n) =&\bk n-\frac{1}{\bk}\sum_{l=1}^nX_{l}\int_{l-1}^l\sin^2\tk(z)\,dz\\
 =&\tk(n-1)+\bk-\frac{X_{n}}{\bk}\int_{n-1}^n\sin^2\tk(z)dz.
\end{split}
\end{eqnarray}
Thus, when $X_n<\bk^2,$ the $\tk$ process increases, $\tk(n+1)>\tk(n),$ while $\tk(n+1)<\tk(n)$ only when $X_n>\bk^2.$ This means for the purpose of obtaining the lower bound on $p(t,s),$ we should use $G$ for $t<s$ and $F$ when $t>s.$ 

So we consider first, $t,s\in A$ with $t<s.$ We need to examine $y\to G(t,y).$
There exists a $y_1\in[\pi,\frac{3\pi}{2}]$  such that  $G(t,y_1)=1.$ For this claim notice that $G(t,\pi)=t<1$ and $y\to G(t,y)$ is increasing on $[\pi, \frac{3\pi}{2})$ with  $\lim_{y\nearrow \frac{3\pi}{2}}G(t,y)=\infty$ and this gives the existence of $y_1.$ Since the function $y\to G(t,y)$ is increasing on $[\pi,y_1],$ when restricted to this interval  its inverse exists, and will be denoted by $G^{-1}(t,\cdot).$  Of course,   $G^{-1}(t,\cdot):[t,1]\to[\pi, y_1].$  Moreover,  we have the derivative in $y$ of $G$ is given by 
\begin{eqnarray}
G_y(t,y)=\frac{(\frac{\bk \tan y}{y})_y(-\frac{y\tan y}{\bk}t+1)+(t+\frac{\bk \tan y}{y})(\frac{y\tan y}{\bk})_yt}{(-\frac{y\tan y}{\bk}t+1)^2}
\end{eqnarray}
and since $y_1\in(\frac{\pi}{2},\frac{3\pi}{2}),$ it's easy to see that there is an $M_\bk<\infty$ such that 
\begin{eqnarray}\label{Gderbd}
\sup_{y\in [\frac{\pi}{2},y_1]} |G_y(t,y)|<M_\bk.
\end{eqnarray}
With this in mind,  notice that
\begin{eqnarray*}
P(\tn(1)\le s|\tn(0)=t)=P(\tn(1)\le s, X_1<\bk^2|\tn(0)=t)+P(\tn(1)\le s, X_1<\bk^2|\tn(0)=t)
\end{eqnarray*}
 and since both terms on the right hand side are increasing we get
 \begin{eqnarray*}
 p(t,s)\ge \frac{d}{ds}P(\tn(1)\le s, X_1<\bk^2|\tn(0)=t).
\end{eqnarray*}
But, 
\begin{eqnarray*}
P(\tn(1)\le s, X_1<\bk^2|\tn(0)=t)=&P(G(t,Z_1)\le s, X_1<\bk^2)\\
=&P(G(t,Z_1)\le s, Z_1\in [t, y_1],\,X_1<\bk^2)\\
+&P(G(t,Z_1)\le s, Z_1\notin [t, y_1],\,X_1<\bk^2)\\
=&P(Z_1\le G^{-1}(t,s), Z_1\in [t, y_1],\,X_1<\bk^2)\\
+&P(G(t,Z_1)\le s, Z_1\notin [t, y_1],\,X_1<\bk^2)\\
=&P(\bk^2-X_1\le G^{-1}(t,s)^2, X_1<\bk^2)\\
+&P(G(t,Z_1)\le s, Z_1\notin [t, y_1],\,X_1<\bk^2)\\
=&\int_{\bk^2- G^{-1}(t,s)^2}^{\bk^2}p(u)du\\
+&P(G(t,Z_1)\le s, Z_1\notin [t, y_1],\,X_1<\bk^2).
 \end{eqnarray*}
 Thus,
 \begin{eqnarray}\label{plowerbd}
p(t,s)\ge 2p(G^{-1}(t,s)^2)\frac{|G^{-1}(t,s)|}{|G_y(t,G^{-1}(t,s)^2)|}.
\end{eqnarray}
Using (\ref{Gderbd}) and (\ref{plowerbd}) and the assumption that $p$ is strictly positive it follows that there is a $d_\bk>0$ such that
\begin{eqnarray}
\inf_{t,s\in A, t>s}p(t,s)\ge d_\bk.
\end{eqnarray}
We now establish the other half of Condition D,
\begin{eqnarray}
\inf_{t,s\in A, t<s}p(t,s)\ge d_\bk.
\end{eqnarray}
For this we resort to use of the function $F$ from (\ref{F}). Simple observation comfirms that $F(t,0)=t+\bk, \lim_{y\to\infty}F(t,y)=0$ and for $t$ fixed, $y\to F(t,y)$ is strictly decreasing. Denote its inverse in $y$ by $F^{-1}(t,\cdot).$ Therefore, for $t\in A,$ there exist $y_1(t)<y_2(t)$ such that $F(t, y_1(t))=1$ and  $F(t, y_2(t))=\frac{1}{\sqrt{3}}.$ In addition, for $y$ fixed, $t\to F(t,y)$ is increasing which implies 
\begin{eqnarray}\label{yineq}
0<y_1(1)<y_1(t)<y_2(t)<y_2(\frac{1}{\sqrt{3}})<\infty ,
\end{eqnarray}
for all $t\in [\frac{1}{\sqrt{3}},1].$

To estimate the density for $t>s$ then, we have
\begin{eqnarray*}
P(t_\bk(1)\le s|t_\bk(0)=t)=&P(t_\bk(1)\le s,\,X_1>\bk^2|t_\bk(0)=t)\\
+&P(t_\bk(1)\le s\,X_1<\bk^2|t_\bk(0)=t)
\end{eqnarray*}
and since both terms on the right hand side are increasing in s we conclude that 

 \begin{eqnarray*}
\frac{d}{ds}P(t_\bk(1)\le s|t_\bk(0)=t)\ge&\frac{d}{ds}P(t_\bk(1)\le s,\,X_1>\bk^2|t_\bk(0)=t)
\end{eqnarray*}
The probability we're differentiating on the right hand side can be simplified,
\begin{eqnarray*}
P(t_\bk(1)\le s,\,X_1>\bk^2|t_\bk(0)=t)=&P(F(t,Z_1)\le s,\,X_1>\bk^2)\\
=&P(F(t,Z_1)\le s,\,X_1>\bk^2)\\
=&P(Z_1\le F^{-1}(t,s),\,X_1>\bk^2)\\
=&P(X_1-\bk^2\le F^{-1}(t,s)^2,\,X_1>\bk^2)\\
=&P(\bk^2<X_1\le F^{-1}(t,s)^2+\bk^2)\\
=&\int_{\bk^2}^{ F^{-1}(t,s)^2+\bk^2}p(x)dx.
\end{eqnarray*}
This we have the lower bound for $t,s\in A$ with $t>s,$
\begin{equation}
p(t,s)\ge p(F^{-1}(t,s)^2+\bk^2)\frac{2F^{-1}(t,s)}{F_y(t,F^{-1}(t,s))}.
\end{equation}
By (\ref{yineq}), the range of $F^{-1}(t,s)$ is contained in $[y_1(1),y_2(\frac{1}{\sqrt{3}})]\subset(0,\infty)$ which implies there is a $d_\bk>0$ such that
\begin{eqnarray}
\inf_{t,s\in A,t>s}p(F^{-1}(t,s)^2+\bk^2)\frac{2F^{-1}(t,s)}{F_y(t,F^{-1}(t,s))}\ge d_\bk.
\end{eqnarray}
\end{proof}

We now turn to $\bf{Condition\,\,B}.$
\begin{lemma}
$\bf{Condition\,\,B}$  holds
\end{lemma}
\begin{proof}
Here we take $W(t)=1_{C^c}(t).$ Then $PW(t)=P(\tn(1)\in C^c|\tn(0)=t).$
Note that $PW(t)\le 1$ for all $t$ so we can take $K=1.$ 
The other requirement splits into two cases.
In the case $t<\frac{\sqrt{3}}{3}$ we have for $Z_1=\sqrt{\bk^2-X_1}1_{\{\bk^2>X_1\}}.$ 
Since $X_1$ has a nonvanishing density $p$ we have for all $t<\frac{\sqrt{3}}{3},$
\begin{eqnarray*}
P(\tn(1)\in C|\tn(0)=t)\ge&P(G(t,Z_1)\in[\frac{\sqrt{3}}{3},1];\bk^2>X_1)\\
>&0
\end{eqnarray*}
and again using the fact that $X_1$ has a nonvanishing density $p$
\begin{eqnarray*}
\lim_{t\to-\infty}P(G(t,Z_1)\in[\frac{\sqrt{3}}{3},1];\bk^2>X_1)=&P(\frac{-1}{\frac{\sqrt{\bk^2-X_1}}{\bk}\tan \sqrt{\bk^2-X_1}}\in [\frac{\sqrt{3}}{3},1];\bk^2>X_1)\\
>&0.
\end{eqnarray*}
Thus, there is an $\epsilon>0$ such that for all $t<\frac{\sqrt{3}}{3},$
\begin{eqnarray}
P(\tn(1)\in C|\tn(0)=t)>\epsilon.
\end{eqnarray}
\end{proof}
In the second case where $t>1,$ taking $Z_1=\sqrt{X_1-\bk^2}1_{\{\bk^2<X_1\}}$
\begin{eqnarray*}
P(\tn(1)\in C|\tn(0)=t)\ge&P(F(t,Z_1)\in[\frac{\sqrt{3}}{3},1];\bk^2>X_1)\\
>&0.
\end{eqnarray*}
Using the fact that $X_1$ has a nonvanishing density $p$
\begin{eqnarray*}
\lim_{t\to-\infty}P(F(t,Z_1)\in\left[\frac{\sqrt{3}}{3},1\right];\bk^2>X_1)=&P(\frac{1}{\frac{\sqrt{X_1-\bk^2}}{\bk}\tan \sqrt{X_1-\bk^2}}\in [\frac{\sqrt{3}}{3},1];\bk^2>X_1)\\
>&0.
\end{eqnarray*}
Thus, taking another value of $\epsilon$ if necessary on has
\begin{eqnarray*}
P(\tn(1)\in C|\tn(0)=t)\ge\epsilon.
\end{eqnarray*}
Thus,
\begin{eqnarray*}
PW(t)\le(1-\epsilon)W(t),\,t\notin C
\end{eqnarray*}
and Condition B is proved.
\end{proof}

\begin{proof}(Theorem\,\ref{MI})
We first establish the existence of the integrated density of states. 



The limiting joint distribution of the Markov chain $\big(\Sk(n):=(X_{n+1},\tilde{\tk}(n))$ is 
\[p(y)\mu_\bk(d\psi)\,dy \]
since the components $X_{n+1}$ and $\tilde{\tk}(n)$ are independent.

This implies that the Markov sequence of processes $$(X_{n+1},\,\{\tilde{\tk}(n+t):0\le t\le 1\}_{n=0}^\infty)$$ has a limiting distribution  on the space $[0,\infty)\times C([0,1],[0,\pi]).$

The Ricatti equation (\ref{ricatti}), 
 \begin{eqnarray}
 \begin{split}
 \tk(n)=\tk(n-1)+\bk-\frac{X_{n}}{\bk}\int_{n-1}^n\sin^2\tk(z)dz.
\end{split}
\end{eqnarray}
implies 
\[\tk(n)-\tk(n-1)=\bk-H(X_{n=1},\tilde{\tk}(n))\]
where, due to smooth dependence on initial date, $H$ is a bounded smooth function on $[0,\infty)\times [0,\pi].$

Therefore,  by the ergodic theorem for Markov chains, we derive the existence of the integrated density of states,
\begin{eqnarray}
\begin{split}
N(\bk^2)=&\frac{1}{\pi}\lim_{n\to\infty}\frac{\tk(n)}{n}\\
=&\frac{1}{n}\sum_{j=0}^{n-1}(\tk(j)-\tk(j-1))\\
=&\bk-\frac{1}{n}\sum_{j=0}^{n-1}H(X_{n=1},\tilde{\tk}(n))\\
=&\bk-\int_0^\pi\int_0^\infty H(x,\psi)p(x)\mu_\bk(\psi)dx.
\end{split}
\end{eqnarray}
It follows easily that,
\[\lim_{\bk\to\infty}\frac{N(\bk^2)}{\bk}= \frac{1}{\pi}.\]

Next we establish the existence of the random Lyapunov exponent in the nonlinear scale $n^{1/\alpha}.$ First we remark that the stationary distribution $\mu_\bk$ has a bounded density. This follows from the smoothness of $H$ above and the fact that the density $p$ is bounded. These imply that the transition density $p(s,t)$ is bounded. Thus, for measurable$A\subset[0,\pi],.$
\begin{eqnarray*}
\mu_\bk(A)=&\int_0^\pi\mu_\bk(d\psi)p(\psi,A)\\
\le& ||p(\cdot,\cdot)||_\infty |A|.
\end{eqnarray*}
Use $\ak(x)$ as above to denote the one-step monodromy matrix defined at (\ref{trans1}) and (\ref{trans2}) for the cases $x\le \bk^2$ and $x>\bk^2,$ respectively. As shown above, the process
$\{\tilde{\theta}_\bk(n):n\geq0\}$ and $\{\Sk(n):=(X_{n+1},\tilde\tk(n)):n\ge 0\}$ are Markov chains with nice stationary distributions
$\mu_\bk(d\theta)$ and $\pi_\bk(dx,d\theta)=p(x)\mu_\bk(d\theta)dx,$ respectively. 
This allows us to apply Furstenberg's Theorem in the argument below.

For ${\bf{v}}=(v_1,v_2)\not={\bf{0}},$ write $\theta_\bk(0)=\frac{v}{||v||}$ and $\tk(j)= \frac{\mk([0,j]{\bf{v}}}{ ||\mk([0,j]{\bf{v}}||}$ where we identify these quantities by means of (\ref{prufer})  with $\mk([0,j]){\bf{v}}$ replacing $(\psi_\bk(j),\psi_\bk'(j))$ so that 
\begin{equation*}
\begin{split}
\ln\frac{||\mk([0,n])\,\bv||}{||\bv||}&=\ln\frac{||\ak(X_n)\mk([0,n-1])\,\bv||}{||\mk([0,n-1])\,\bv||}+\ln\frac{||\ak(X_{n-1})\mk([0,n-2])\,\bv||}{||\mk([0,n-2])\,\bv||}\\
& \quad+\dotsb+\ln\frac{||\ak(X_1)\,\bv||}{||\,\bv||}\\
&=\sum_{j=0}^{n-1}\ln||\ak(X_{j+1})\theta_k(j)||\\
 &=\sum_{j=0,\,X_j\ge \bk^2}^{n-1}\ln||\ak(X_{j+1})\theta_k(j)||+\sum_{j=0,\,X_j<\bk^2}^{n-1}\ln||\ak(X_{j+1})\theta_k(j)||\\
 &=I_n\,+\,II_n.     
    \end{split}
  \end{equation*}
  so $\ln||\mk([0,n])\theta||$ is an additive
  functional of the Markov chain $\{\Sk(n):n\ge0\}.,$
  \begin{equation*}
    \ln\frac{||\mk([0,n])\bv||}{||\bv||}=\sum_{j=0}^{n-1}f(\Sk(j))
  \end{equation*}
  with
  \begin{equation*}
    f(x,{\bf{w}})=\ln\frac{||A(x){\bf{w}}||}{||{\bf{w}}||}.
  \end{equation*}
  By Furstenberg's Theorem,
  \[\lim_{n\to\infty}\frac{1}{n}II_n=\eta>,\,a.s.\]
  and consequently,
  \[\lim_{n\to\infty}\frac{1}{n^{1/\alpha}}II_n=0,\,a.s..\]
  By our assumption on the density of the iid sequence $\{X_j:j\ge1\},$
  \[\lim_{n\to\infty}\frac{1}{n}\lt|\{j\le n:\,X_j>\bk^2\}\rt|=\int_{\bk^2}^\infty p(x)dx>0.\]
Thus we only need show that $\frac{1}{n^{1/\alpha}}I_n$ has a limiting $St_\alpha$ distribution.   According to a general result in \cite{JKO} about limits of additive
  functionals, the claim of the Theorem follows if we establish that
  \begin{equation*}
    \int_{f(x,\theta)>z}\pi_\bk(dx,d\theta)\sim c_1z^{-2\alpha}\quad \text{as $z\to\infty$}
  \end{equation*}
  which translates to
  \begin{equation*}
    \int_{||A(x)\theta||>e^z||\theta||}p(x)\mu_\bk(d\theta)dx\sim c_1z^{-2\alpha}\quad \text{as $z\to\infty$}.
  \end{equation*}
  If $x>\bk^2,$ the matrix $\ak(x)$ has two positive eigenvalues
  $$t_{\pm}(x)=e^{\pm\sqrt{x-\bk^2}}$$
  with corresponding eigenvectors
  $$v_{\pm}=\left(1,\pm\frac{y}{\bk}\right)^T.$$
  so $\ln t_+(X_n)=\sqrt{X_n-\bk^2}$ belongs to the domain of
  attraction of the stable law of index $\alpha$.  Also, the
  matrix $\Dk(x):=\big(\ak(x)^*\ak(x)\big)^{1/2}$ has eigenvalues
  $$ \mu_{\pm}(x)=\frac{\sqrt{b+2}\pm\sqrt{b-2}}{2}$$
  with $b=2\cosh^2a+(a^2+1/a^2)\sinh^2a$ and $a=\sqrt{x-\bk^2}$.
  So $\mu_+(x)$ can be written as
  $$ \mu_+(x)=\left(a+\frac{1}{a}\right)|\sinh a|\frac{1+\sqrt{1+\frac{4}{\left(a+\frac{1}{a}\right)\sinh^2 a}}}{2}$$
  and since
  $$ \frac{\ln\mu_+(x)}{a}\to 1\quad\text{as $a\to\infty$}$$
  it follows that $\log\mu_+(X_n)$ also belongs to the domain of attraction of the
  stable law of index $\alpha$.

  Now
  $$ \int_{||A(x)\theta||>e^z||\theta||}p(x)\mu_\bk(d\theta)dx= \int_{\mu_+(x)>z}\mu_\bk(E_x)p(x)dx$$
  with $E_x:=\{\theta\colon ||\Dk(x)\theta||>e^z||\theta||\}$.
  Since $\Dk(x)$ is selfadjoint, the set $E_x$ is in fact a cone centered
  at the eigenvector of $\mu_+$.  Its Lebesgue measure can be computed
  explicitly and is
  $$g(x,z)=2\arctan\sqrt{\frac{e^{2(\mu_+(x)-z)}-1}{1-e^{-2(\mu_+(x)+z)}}}.$$
  Since $\mu_\bk(d\theta)$ has a density $m_\bk(\theta)$ w.r.t. Lebesgue
  measure, satisfying $m\leq m_\bk(\theta)\leq M$, we have
  $$m\int_{\mu_+(x)>z}g(x,z)p(x)dx\leq\int_{\mu_+(x)>z}\mu_\bk(E_x)p(x)dx
  \leq M\int_{\mu_+(x)>z}g(x,z)p(x)dx$$
  therefore the proof is completed with the application of the
  following lemma, with
  $$h(u,z)=2\arctan\sqrt{\frac{e^{2(u-z)}-1}{1-e^{-2(u+z)}}}$$
  and $U=\mu_+(X_1)$.
\end{proof}

\begin{lemma}
  If $U>0$ is a r.v. in the domain of attraction of the stable
  law of index $0<\alpha<2$, and $h(u,z)$ is a function satisfying
  \begin{eqnarray}
  \begin{split}
    \lim_{u\to+\infty}h(u,z)=k,\quad\forall z\\
    \lim_{z\to+\infty}h(tz,z)=k,\quad\forall t>1\label{cond:h}
  \end{split}
  \end{eqnarray}
  then there is a $k'>0$ such that
  $$\mathbb{E}[h(U,z);U>z]\sim k'z^{-\alpha}\quad\text{as $z\to+\infty$}.$$
\end{lemma}
\begin{proof}
  Let $F$ be the distribution function of $U$, so that
  $z^{\alpha}\big(1-F(z)\big)=c\big(1+o(1)\big)$ as $z\to\infty$.
  Then
  \begin{equation*}
    \begin{split}
      z^{\alpha}\mathbb{E}[h(U,z);U>z]&=z^{\alpha}\int_{u>z}h(u,z)dF(u)\\
      &=z^{\alpha}\int_{t>1}h(tz,z)dF(zt),\quad t=\frac{u}{z}.
    \end{split}
  \end{equation*}
  By \eqref{cond:h} we have $\forall \epsilon>0$
  $$z^{\alpha}\int_{t>1}(k-\epsilon)dF(zt)<z^{\alpha}\int_{t>1}h(tz,z)dF(zt)<z^{\alpha}\int_{t>1}(k+\epsilon)dF(zt)$$
  for $z$ large enough, which going back to the variable $u$ becomes
  $$z^{\alpha}\int_{u>z}(k-\epsilon)dF(u)<z^{\alpha}\int_{u>z}h(u,z)dF(u)<z^{\alpha}\int_{u>z}(k+\epsilon)dF(u).$$
  Hence the statement of the Lemma with $k'=kc$.
\end{proof}  

\begin{proof}
We now prove Theorem (\ref{MII}) for ${\bf{Model\,\,II}}.$
In view of (\ref{trans3}), for $n\le x\le n+1,$ we can write when $X_l+\lambda\ge 0.$
\begin{eqnarray}
\begin{split}
\left(\begin{array}{l}\psi_\lambda(x)\\
\frac{1}{\lambda}\psi_\lambda'(x)\end{array}\right)
=&
\left(\begin{array}{ll}\cos(\sqrt{\lambda+X_n}(x-n))&\frac{|\lambda|}{\sqrt{\lambda+X_n}}\sin(\sqrt{\lambda+X_n}(x-n))\\
-\frac{\sqrt{\lambda+X_n}}{\lambda}\sin(\sqrt{\lambda+X_n}(x-n))&\cos(\sqrt{\lambda+X_n}(x-n))\end{array}\right)\\\times&\left(\begin{array}{l}\psi_\lambda(n)\\
\frac{1}{\lambda}\psi_\lambda'(n)\end{array}\right)
\end{split}
\end{eqnarray}
from which it's clear that the vector 
$$\left(\begin{array}{l}\psi_\lambda(x)\\
\frac{1}{\lambda}\psi_\lambda'(x)\end{array}\right)$$
traces out a curve through a total angle of $\sqrt{\lambda+X_l}.$ That is
\begin{eqnarray*}
\theta_\lambda(n+1)=\theta_\lambda(n)+\sqrt{\lambda+X_n}.
\end{eqnarray*}
Thus,
\begin{eqnarray}\label{tksum}
\theta_\lambda(n)=\theta_\lambda(0)+\sum_{l=1}^n\sqrt{\lambda+X_l}
\end{eqnarray}
Due to the assumption on the distribution of the $\{X_l\}$ it follows that
\begin{eqnarray}
\lim_{n\to\infty}\frac{\theta_\lambda(n)}{\pi\,n^{1/\alpha}}\stackrel{\mathcal{L}}{=}\zeta_\alpha
\end{eqnarray}
with $\zeta_\alpha$ a random variable with an $St_\alpha$ distribution. This implies
\begin{eqnarray}
\lim_{n\to\infty}\frac{\theta_\lambda(n)}{\pi\,n}=\infty,\,a.s..
\end{eqnarray}
From (\ref{tksum}) it follows that $\theta_\lambda(n)\mod\pi$ has a nice density. Also, by the assumption on the distribution of $X_n$ it follows that $E[\ln^+||\mathcal{M}_\lambda(n)||]<\infty$ and so by standard results, $\gamma(\lambda)$ exists $P\, a.s.$ and is positive. From Kotani's results, \cite{K}, we can then conclude the spectrum is pure point $a.s.$ for $a.e.\,\,\theta_0$ since the resolvent in this case is in $L^2.$ The growth rate for the eigenfunctions also follows by Kotani's method.
\end{proof}
\begin{proof}
We now prove Theorem (\ref{MIII}) for $(\bf{Model\,\,III}).$
Since the random variables $Y_l$ possess a density, the Ricatti equation (\ref{ricatti}) implies that the distribution of the phase $\tk$ has a density (note that the distribution of $\bk Y_l\mod \pi$ has a bounded density) and since $\mk([0,L_n])=\Ck(Y_n)\cdots\Ck(Y_2) \Ck(Y_1)$, with $E[\ln^+|| \Ck(Y_1)||]<\infty,$ the Furstenberg Theorem holds:
\[\lim_{n\to\infty}\frac{1}{n}\ln||\mk([0,L_n])||=\gk_{nl}>0,\,P-a.s..\] 
The Lyapunov exponent with respect to the linear scale then satisfies
\begin{eqnarray}
\gk=\lim_{n\to\infty} \frac{1}{n}\ln ||\mk([0,n])||=0,\,a.s..
\end{eqnarray}
This follows easily from the fact that $L(n)\sim n^{1/\alpha},$
\begin{eqnarray*}
\lim_{n\to\infty} \frac{1}{n}\ln ||\mk([0,n])||=&\lim_{n\to\infty} \frac{\ln ||\mk([0,n])||}{\ln ||\mk([0,L_n])||} \frac{\ln ||\mk([0,L_n])||}{n}\\
=&\gk_{nl}\lim_{n\to\infty} \frac{\ln ||\mk([0,n])||}{\ln||\mk([0.L_n])||} \\
=&0.
\end{eqnarray*}
Finally, observe that
\begin{eqnarray*}
\lim_{n\to\infty} \frac{\ln ||\mk([0,L_n])||}{L_n^\alpha}=&\lim_{n\to\infty} \frac{\ln ||\mk([0,L_n])||}{n}\frac{n}{L_n^\alpha}\\
=&\frac{\gk_{nl}}{\zeta_\alpha^\alpha},
\end{eqnarray*}
where $\zeta_\alpha$ again has a $St_\alpha$ distribution.

We consider now the asymptotic behavior of the rotation number $\tk(n).$
By the Ricatti equation (\ref{ricatti}),  across the $n^{th}$ gap the rotation is
\[\tk(L_n-1)-\tk(L_{n-1})=\bk\,Y_n,\]
whereas across the $n^{th}$ bump the rotation is
\[\tk(L_n)-\tk(L_n-1)=O(1).\]
Thus,
\begin{eqnarray}\label{rot3}
\tk(L_n)=\bk\sum_{j=1}^n Y_j+O(n)
\end{eqnarray}
which implies by convergence to the stable law $St_\alpha$
\[\lim_{n\to\infty}\frac{\tk(L_n)}{n^{1/\alpha}}=\bk \zeta_\alpha.\]
But also from (\ref{rot3}) follows
\[\lim_{n\to\infty}\frac{\tk(L_n)}{L_n}=\bk\]
and consequently,
\[N(\bk^2)=\frac{\bk}{\pi}.\]

We turn now to consideration of the spectrum.  Since $\ln||M_\bk([0,L_n])||\sim {\gk} n, $ and by (\ref{ricatti}) the magnitude of $r_\bk(x)$ is constant across the gaps, there are two so-called Weil solutions on $[0,\infty),$  one of which, $\psi_\bk^+,$ has the maximal rate of growth i.e.
 \begin{eqnarray}
\psi_\bk^+(L_n)\le C e^{(\gk+\epsilon)n},\,n\ge n_0(\omega),\, P-a.s..
\end{eqnarray}
The second solution, $\psi_\bk^-,$  is exponentially decreasing, i.e.
\begin{eqnarray}
\psi_\bk^-(L_n)\le C e^{(-\gk+\epsilon)n},\,n\ge n_0(\omega), \, P-a.s..
\end{eqnarray}
Note that by the Ricatti equation (\ref{ricatti}), the function $r_\bk$ is constant on the intervals between bumps.
The Green function, $R_{\bk+i0}(0,L_n)$ in this case, is decreasing exponentially with rate at least $\gk-\epsilon.$
Thus, $P-a.s.,$
 \begin{eqnarray}\label{R-est}
 \int_0^\infty |R_{\bk+i0}(0,x)|^2dx\le \sum_{n=0}^\infty Y_ne^{-2(\gk-\epsilon)n}+O(1).
 \end{eqnarray} 
 The fluctuations of the sequence $Y_n$ can be estimated by noting 
 \[P(Y_n>n^{1/\alpha} \ln^{2/\alpha}n)\sim \frac{c}{n\ln^2 n},\]
 so that by Borel-Cantelli,
  \[P(Y_n>n^{1/\alpha} \ln^{2/\alpha}n,\,i.o.)=0.\]
 This implies by (\ref{R-est}) that  $a.s.\,R_{\bk+i0}(0,\cdot)\in L^2([0,\infty)).$ Using this with Kotani's result and the absolutely continuous distribution of the phase $\tk(S_n),$ give the localization result, namely, for a.e. $\theta_0,$ the spectrum of $H^{\theta_0}$ is pure point  $a.s..$
  \end{proof}

\begin{remark}
We wish to stress again that the fact that the spectrum in ${\bf{Model\,III}}$ is pure point while the Lyapunov exponent is identically zero doesn't contradict the "Kotani Theory." The reason of course is that the point potential in ${\bf{Model\,III}}$ is not ergodic.
\end{remark}



\end{document}